\documentclass[11pt]{article}
\usepackage{amsthm, amsmath, amssymb, amsfonts, url, booktabs, tikz, setspace, fancyhdr, bm}
\usepackage[hidelinks]{hyperref}
\usepackage{geometry}
\usepackage{hyperref, enumerate}
\usepackage[shortlabels]{enumitem}
\usepackage[babel]{microtype}
\usepackage[english]{babel}
\usepackage[capitalise]{cleveref}
\usepackage{comment}
\usepackage{bbm,tkz-graph,subcaption}
\usepackage{csquotes}
\usepackage{mathrsfs}
\usepackage{mathabx}
\usetikzlibrary{patterns}
\usetikzlibrary{shapes}
\usetikzlibrary{decorations.pathreplacing}
\usetikzlibrary{decorations.pathmorphing}
\usetikzlibrary{snakes}
\usepackage{bbm}

\usepackage[T1]{fontenc}
\usepackage{amsfonts}
\usepackage{amscd}
\usepackage{graphicx}
\usepackage{enumitem}
\usepackage{verbatim}
\usepackage{hyperref}
\usepackage{amsmath,caption}
\usepackage{url,pdfpages,xcolor,framed,color}
\usepackage{todonotes}
\usepackage{comment}

\newtheorem{thr}{Theorem}[section]

\newtheorem{lem}[thr]{Lemma}
\newtheorem{prop}[thr]{Proposition}

\theoremstyle{definition}
\newtheorem{defi}[thr]{Definition}
\newtheorem*{defi*}{Definition}

\newtheorem{prob}[thr]{Problem}

\def\C{\mathcal{C}}

\def\E{\mathbb{E}}
\def\eps{\varepsilon}

\newcommand*{\abs}[1]{\lvert #1\rvert}
\newcommand*{\bceil}[1]{\left\lceil #1\right\rceil}
\def\a{\alpha}
\def\b{\beta}

\title{Many Hamiltonian subsets in large graphs with given density}
\date{}

\author{
Stijn Cambie \thanks{Extremal Combinatorics and Probability Group (ECOPRO), Institute for Basic Science (IBS), Daejeon, South Korea, supported by the Institute for Basic Science (IBS-R029-C4),
E-mail: {\tt stijn.cambie@hotmail.com or  \{stijncambie, jungao, hongliu\}@ibs.re.kr}.}
\and
Jun Gao\footnotemark[1] 
\and
Hong Liu\footnotemark[1] 
}

\begin{document}
\maketitle

\begin{abstract}
    A set of vertices in a graph is a Hamiltonian subset if it induces a subgraph containing a Hamiltonian cycle. Kim, Liu, Sharifzadeh and Staden proved that among all graphs with minimum degree $d$, $K_{d+1}$ minimises the number of Hamiltonian subsets. We prove a near optimal lower bound that takes also the order and the structure of a graph into account. For many natural graph classes, it provides a much better bound than the extremal one ($\approx 2^{d+1}$). Among others, our bound implies that an $n$-vertex $C_4$-free graphs with minimum degree $d$ contains at least $n2^{d^{2-o(1)}}$ Hamiltonian subsets.
\end{abstract}

\section{Introduction}
Finding sufficient conditions that guarantee the existence of certain cycles is a well-studied topic in combinatorics.
A cycle in a graph is \emph{Hamitonian} if it spans the whole vertex set of the graph. Testing whether a graph contains a Hamiltonian cycle is one of Karp's original NP-complete problems \cite{karp1972reducibility}. Dirac's theorem~\cite{dirac1952some} from $1952$, arguably the most influential result in this area, asserts that $\delta \ge \frac n2$ is a tight minimum degree condition for containing a Hamiltonian cycle.
Since then, various extensions have been studied over the past 70 years, see e.g.~\cite{csaba2016proof,knox2015edge,kuhn2013hamilton,kuhn2014hamilton,Kuhn2005} and the survey \cite{kuhn2014hamiltonsurvey}. 

In this paper, we study the enumeration problem on Hamiltonian subsets of a graph. A set of vertices $A \subseteq V(G)$ is a \emph{Hamiltonian subset} if $G[A]$ contains a Hamiltonian cycle. Denote by $h(G)$ the number of Hamiltonian subsets of $G$. It is natural to ask how $h(G)$ relates to the minimum degree. Intuitively, when the minimum degree is given, larger graphs tend to have more Hamitonian subsets. In 1981, Koml\'os conjectured that among all graphs with minimum degree at least $d$, the complete graph $K_{d+1}$ minimises the number of Hamiltonian subsets. This conjecture was recently confirmed for large $d$ by Kim, Liu, Sharifzadeh and Staden~\cite{KLSS17}, who also showed that $K_{d+1}$ is the \emph{unique} minimiser.

While~\cite{KLSS17} brings a happy ending to Koml\'os's conjecture, it leaves much to be desired. Perhaps the most natural question, considering the $n$-vertex graph $G^{\star}$ consisting of $\frac{n-1}{d}$ copies of $K_{d+1}$ sharing exactly one common vertex, is whether for any $n$-vertex graph $G$ with $\delta(G)=d$, $h(G)=\Omega(n2^d)$. Also, notice that $G^{\star}$ is basically a disjoint union of $K_{d+1}$s. If components of a graph are much larger than the unique minimiser $K_{d+1}$, is it possible to obtain an exponential improvement on $2^d$? As we shall see, the relevant parameter is the `essential order' of a graph, captured by the following notion of \emph{crux} introduced by Haslegrave, Hu, Kim, Liu, Luan and Wang~\cite{HHKLLW22}. Roughly speaking, the crux of a graph is large when the edges are relatively uniformly distributed. We write $d(G)$ for the average degree of $G$.

\begin{defi*}[Crux]
    For a constant $\a \in (0,1)$, a subgraph $H\subseteq G$ is an \emph{$\a$-crux} if $d(H)\ge \a \cdot d(G)$. 
    The $\a$-crux function, $c_\a(G)$, of $G$ is defined as the order of a minimum $\a$-crux in $G$, that is,
    $$ c_\a (G)=\min \{ \abs H \colon H \subseteq G \text{ and } d(H) \ge \a \cdot d(G) \}.$$
\end{defi*}

Here is our main result. 

 \begin{thr}\label{main:1}
   There exist constants $B$ and $d_0$ such that the following is true. Let $G$ be an $n$-vertex graph with average degree $d\ge d_0$, $t=c_{\frac{1}{5}}(G)$ and $\b=(6000\log3)^{-16}$, then
    $$h(G)\ge \frac{1}{B}n2^{\b t/\log^{16}t}.$$
\end{thr}

Our bound is optimal up to the constant factor $B$ and the polylog factor in the exponent (consider again $G^{\star}$). Also $\frac 15$ in the crux function can be replaced by any constant strictly smaller than $\frac 12.$
It improves on the extremal bound $h(K_{d+1})\approx 2^{d+1}$ in two aspects as suggested above, i.e. having a factor linear in the order $n$ and an exponential (in $d$) improvement for graphs whose crux size is much larger than their average degree. For example, for a $C_4$-free graph $G$, we get $h(G)\ge n2^{d^{2-o(1)}}$. Here are some common graph classes for which $c_{\alpha}(G)\gg d(G)$: $(i)$ $K_{s,t}$-free graphs $G$ with $s,t\ge 2$, satisfy $c_\alpha(G)=\Omega\big(d(G)^{s/(s-1)}\big)$; $(ii)$ a $\frac{d}{r}$-blow-up $G$ of a $d$-vertex $r$-regular expander graph for a sufficiently large constant $r$ satisfies $c_\alpha(G)=\Omega(d^2)$ and $d(G)=d$; and $(iii)$ graph products have even exponential gap, e.g. using isoperimetry inequalities, one can show that the $d$-dimension hypercube $Q^d$ satisfies $c_\alpha(Q^d)\ge 2^{\alpha d}$.

It is worth mentioning that Theorem~\ref{main:1} is another manifestation of the \emph{replacing average degree by crux} paradigm proposed in Haslegrave, Hu, Kim, Liu, Luan and Wang~\cite{HHKLLW22}. It suggests that one might be able to replace the appearance of $d(G)$ in results on sparse graph embeddings with the crux size $c_\alpha(G)$ instead. We refer the readers to~\cite{HHKLLW22, Im22} for more results illustrating this paradigm.

Our proof takes a different approach than that in~\cite{KLSS17}. To have the factor linear in the order $n$, we find a positive fraction of vertices, each contained in many distinct Hamiltonian subsets. To this end, we repeatedly apply the following result, which guarantees one such `heavy' vertex, in different dense subgraphs of the host graph.

\begin{thr}\label{cor:cycles with v}
    Let $0<\alpha<\frac{1}{2}$, $G$ be a graph with sufficiently large average degree $d$ and $t=c_{\alpha}(G)$. Then there exists a vertex lying in at least $2^{\b t/\log^{16} t}$ distinct Hamiltonian subsets, where $\b=(6000\log3)^{-16}$.
\end{thr}

We find `heavy' vertices via embedding a large wheel-like structure (see Definition~\ref{defn:wheel}), inspired by the adjuster structure in Liu-Montgomery~\cite{LM22}, in sublinear expander subgraphs of $G$. To construct such large wheel, we perform an exploration algorithm similar to Depth First Search on a collection of suitable cycles. The theory of sublinear expanders, first introduced by Koml{\'o}s and Szemer{\'e}di~\cite{KS94,KS96} in the 1990s, has played a pivotal role in many recent resolutions of old conjectures; see e.g.~\cite{FKKl22,FL22,Fer22,HHKLLW22,HHKL22,HKL22,KLSS17,LM17,LM22,LWY22}.

It would be interesting to know whether the polylog factor in the exponent in Theorem~\ref{main:1} is necessary. 
Regular expanders can have girth polylogarithmic in $n$, so our approach cannot be applied to delete the polylog factor. On the other hand, we observe that the polylog factor is not necessary for $(n,d,\lambda)$-graphs. The $(n,d,\lambda)$-graphs are $d$-regular graphs on $n$ vertices with second largest eigenvalue in absolute value $\lambda$. It is not hard to see that the crux size of an $(n,d,\lambda)$-graph is linear in $n$.

\begin{prop}\label{coro:3}
    Let $0<\beta <1$ and $G$ be an $(n,d,\lambda)$-graph. 
    If  $\frac{d}{|\lambda|} \ge \frac{1}{\b ^2}$,
    then $ h(G) \ge  \binom{n}{\frac{1}{2}n} / \binom{(\frac{1}{2}+ 3 \b)n}{3\b n}$. 
    Specially, if $\b < 1/6$, then $h(G) \ge  2^{(\frac{1}{2}-3\b)n}$.
\end{prop}

\medskip

\noindent\textbf{Organisation.} The rest of the paper is organised as follows. In Section~\ref{sec:not}, we list some preliminaries needed for the proof. In Section~\ref{sec:manydiscycles}, we prove Theorem~\ref{cor:cycles with v}. In Section~\ref{sec:mainthr}, we prove the main result, Theorem~\ref{main:1}. Proposition~\ref{coro:3} is proved in Section~\ref{sec:proofbgraph} and concluding remarks in Section~\ref{sec:conc}.

\section{Notations and preliminary properties}
\label{sec:not}

A \emph{ball of radius $r$ (around a vertex $v$)}, denoted by $B_G^{r}(v)=\{u \in V \colon 0 \le d(u,v) \le r\}$, in a graph $G$ is the set of all vertices which are at distance no more than $r$ from $v$. Here we write $B^r(v)$ if the underlying graph we consider is clear.
For a vertex set $X,$ the ball around $X$ of radius $r$ is similary defined as the set of all vertices at distances at most $r$ from (some vertex in) $X$, that is $B^{r}(X)=\bigcup_{v \in X}B^{r}(v)$. We write $G-X=G[V(G)\setminus X]$ for the subgraph induced on $V(G)\setminus X$. Throughout the paper, $\log$ denotes the natural logarithm.

\subsection{Sublinear expanders}
For $\eps_1 >0$ and $k>0$. Let $\eps(x,\eps_1,k)$ be the function

\begin{align}\label{formula:eps}
\eps(x,\eps_1,k) = \left\{
    \begin{array}{lr}
         0 &\text{if}\  x<k/5,  \\
         \eps_1/\log^2(15x/k) &\text{if}\  x\ge k/5, 
    \end{array}
    \right.
\end{align}
where, when it is clear from context, we will write $\eps(x,\eps_1,k)$ as $\eps(x).$

\begin{defi}[Sublinear expander]
A graph $G$ is an \emph{$(\eps_1, k)$-expander} if for any subset $X \subseteq V(G)$ of size $k/2 \le |X| \le |V(G)|/2$, we have $|N_G(X)| \ge \eps (|X|)\cdot |X|$.
\end{defi}

A classical result of Koml{\'o}s and Szemer{\'e}di states that any graph contains a sublinear subgraph retaining almost the same average degree.
\begin{lem}[Lemma~2.2, \cite{HHKLLW22}]\label{lem:existence_expander}
    Let $C>30,\eps_1\le1/(10C),k > 0$ and $d > 0$.
    Then every graph $G$ with $d(G) = d$ has a subgraph $H$ such that $H$ is an $(\eps_1,k)$-expander, 
    $d(H) \ge (1-\delta)d$ and $\delta (H) \ge d(H)/2$, where $\delta := \frac{C \eps_1}{\log3}$.
\end{lem}

The following lemma is a slight modified version of Theorem~3.12 in~\cite{LM22}. It finds linear-size balls robustly in expanders. 

\begin{lem}\label{lem:Balln/10} For any $0<\eps_1<1$ the following holds for each $n\geq 60$. Suppose that $G$ is an $n$-vertex $(\eps_1,15)$-expander.
For any set $W\subseteq V(G)$ with $|W|\leq \eps_1 \frac{n}{20\log^{2}n}$, there is a ball $B\subseteq G-W$ with size at least $n/10$ and radius at most $\frac{20}{\eps_1}\log^3n$.
\end{lem}

\begin{lem}[Proposition~3.10, \cite{LM22}]\label{lem:subgraphBall}
For every $m'\le m$ the following is true.
    Every ball $B^r(v)$ of radius $r$ with size $m$ contains a connected subgraph of radius at most $r$ with center $v$ and size $m'.$ 
\end{lem}

A key property of expanders is the following short diameter property.

\begin{lem}[\cite{KS96}]
\label{lem:avoidWshortpath}
	Let $\eps_1>0$, $H$ be an $n$-vertex $(\eps_1,15)$-expander and $X,X',W\subseteq V(H)$. If $|X|,|X'|\ge x\ge 8$ and $|W|\le \frac{1}{4}\eps(x)x$, then there is a path in $H-W$ from $X$ to $X'$ of length at most $\frac{2}{\eps_1}\log^3n$.
\end{lem}

\subsection{Large wheels in expanders}
To find vertices in many Hamiltonian subsets, we use the following structures. 

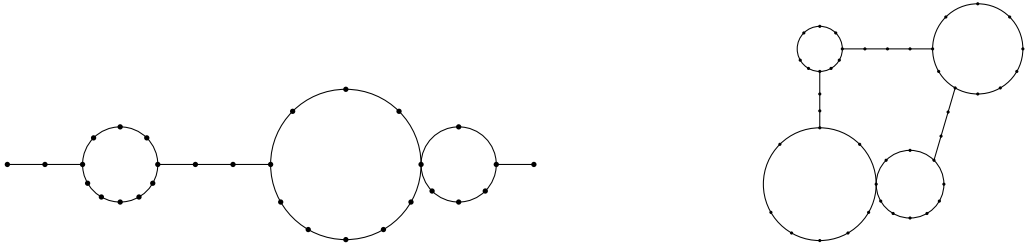
\begin{figure}[htbp]
\centering
\begin{minipage}[b]{0.48\textwidth}
\centering
\begin{tikzpicture}[scale=0.5]

\draw (0,0) circle (1);
\draw (6,0) circle (2);
\draw (9,0) circle (1);
\foreach \i in {-1,-2,-3,1,2,3,4,8,10,11}{
\fill (\i,0) circle (2pt);
}
\draw [-] (1,0)--(4,0);
\draw [-] (-1,0)--(-3,0);
\draw [-] (10,0)--(11,0);

\foreach \x in {-45,0,45,120,150,180,210,240}{
\fill ({sin(\x)},{cos(\x)}) circle (2pt);
}
\foreach \x in {-45,0,45,120,150,180,210,240}{
\fill ({2*sin(\x)+6},{2*cos(\x)}) circle (2pt);
}

\foreach \x in {0,135,180,225}{
\fill ({sin(\x)+9},{cos(\x)}) circle (2pt);
}

\end{tikzpicture}
\label{figure:chain}
\end{minipage}
\begin{minipage}[b]{0.48\textwidth}
\centering
\begin{tikzpicture}[scale=0.3]
\draw (0,0) circle (2.5);

\foreach \x in {-45,0,45,90,120,150,180,210,240}{
\fill ({2.5*sin(\x)},{2.5*cos(\x)}) circle (2pt);
}

\draw (0,6) circle (1);
\foreach \x in {-45,0,45,90,120,150,180,210,240}{
\fill ({sin(\x)},{cos(\x)+6}) circle (2pt);
}

\draw (4,0) circle (1.5);
\foreach \x in {-45,0,45,90,120,150,180,210,240}{
\fill ({1.5*sin(\x)+4},{1.5*cos(\x)}) circle (2pt);
}

\draw (7,6) circle (2);
\foreach \x in {-45,0,45,90,120,150,180,210,240,270}{
\fill ({2*sin(\x)+7},{2*cos(\x)+6}) circle (2pt);
}

\draw [-] (0,2.5)--(0,5);
\fill (0,3.25) circle (2pt);
\fill (0,4) circle (2pt);

\draw [-] (1,6)--(5,6);
\foreach \n in {2,3,4}{
\fill (\n,6) circle (2pt);}

\draw [-] ({1.5*sin(45)+4},{1.5*cos(45)})--({2*sin(210)+7},{2*cos(210)+6});

\fill ({2*(1.5*sin(45)+4)/3+(2*sin(210)+7)/3},{2*(1.5*cos(45))/3+(2*cos(210)+6)/3}) circle (2pt);
\fill ({(1.5*sin(45)+4)/3+2*(2*sin(210)+7)/3},{(1.5*cos(45))/3+2*(2*cos(210)+6)/3}) circle (2pt);

\end{tikzpicture}
\label{figure:wheel}
\end{minipage}
\caption{A $3$-chain and a $4$-wheel}\label{figure:chain&wheel}
\end{figure}

\begin{defi} [Chain/wheel]\label{defn:wheel}
    An \textbf{$\ell$-chain}/\textbf{$\ell$-wheel} is the graph obtained from a path/cycle by replacing $\ell$ edges in the path/cycle with $\ell$ cycles.
\end{defi}
Examples of a chain and a wheel are depicted in Figure~\ref{figure:chain&wheel}. It is easy to see that an $\ell$-wheel has at least $2^\ell$ different Hamiltonian subsets.

\begin{lem}\label{thr:expander-cycle}
Let $0<\eps_1 <1$ and $H$ be an $(\eps_1, 15)$-expander of order $n$, where $n$ is sufficiently large. Then there exists an $\ell$-wheel in $H$ with $\ell \ge \left(\frac{\eps_1}{20} \right)^{16} (2n/\log^{16} n)$.  
\end{lem}

\subsection{Crux function}
We also need some control on the crux function and the expander mixing lemma.
\begin{lem} [Expander mixing lemma~\cite{AC88}]\label{EML}
   For any $(n,d,\lambda)$-graph $G$ and two vertex subsets $X$ and $Y$, we have $\bigg|e(X,Y) -\frac{d}{n}|X||Y|\bigg|\le   \lambda \sqrt{|X||Y|\left(1-\frac{|X|}{n}\right)\left(1-\frac{|Y|}{n}\right)}$.
\end{lem} 

\begin{prop}\label{prop:ca(G)~an}
    For every graph $G$ with average degree $d$ and every $0 < \alpha< \alpha' < 1$, we have 
    $$c_{\a}(G) \le \bceil{\frac{\a}{\a '}(c_{\a '}(G)-1)+1}.$$
    In particular, for a graph $G$ of order $n$, for every $0<\a<1$, 
     $c_{\a}(G) \le \bceil{\a(n-1)+1}.$
\end{prop}
\begin{proof}
    Let $G'$ be the graph on $n:= c_{\alpha '}(G)$ vertices for which the average degree is at least $\a' d.$
    By definition of the crux, it suffices to prove that there exists a subgraph $H$ of $G'$ with at most $k:=\bceil{\frac{\a}{\a '}(n-1)+1}$ vertices for which the average degree is at least $\a d.$ 
    For this, consider all $\binom{n}{k}$ possible induced subgraphs $H$ of $G'$ with $k$ vertices.
    The probability that a particular edge $uv$ of $G$ belongs to such a subgraph equals $\frac{\binom k2}{\binom {n}{2}}.$
    The expected size of $H$ equals 
    $\E(E(H)) =e(G') \frac{\binom k2}{\binom {n}{2}}.$
    Hence the expected average degree equals
    \begin{align*}
        \E(d(H)) = \frac{2\E(E(H))}{k}
        \ge\frac{\alpha' d n }{k}\frac{\binom k2}{\binom {n}{2}}
        =\frac{ \alpha 'd(k-1)}{n-1}\ge\a d. 
    \end{align*}
    This implies that there exists at least one subgraph $H \subset G$ with order $k$ and $d(H) \ge \a d$.
\end{proof}

We remark (even while not needed in the remaining of the expository) that Proposition~\ref{prop:ca(G)~an} is sharp for $G=K_{d+1},$ since $H=K_{\a d+1}$ is the minimum subgraph of $G$ with average degree at least $\a d.$
More generally, it is asymptotically sharp for $(n,d,\lambda)$-graphs.

\begin{prop}\label{prop:ca(G)}
   Let $0 < \a < 1$. Given $\eps>0$, if $\frac{\lambda}{d}< \eps \a$, then for every $(n,d,\lambda)$-graph $G$ $$c_{\a}(G) >(1-\eps)\a n.$$
\end{prop}

\begin{proof}
  Assume there exists a set $S$ such that $d(G[S]) \ge \a d$ and $\abs{S}<(1-\eps)\a n.$ Then $$\abs{e(S, V \backslash S)} \le (1-\a) d \abs{S}   = (1-\a+\a \eps) d \abs{S} - \a \eps d \abs{S}<  d \abs{S} \frac{\abs{V \backslash S}}{\abs{V}} - \lambda \abs{S}.$$ This is a contradiction with the expander mixing lemma.
\end{proof}

\subsection{Depth First Search}
One of the main ideas in this paper is an algorithm that is similar to Depth First Search (DFS); see step 2 in Section~\ref{sec:manydiscycles}). DFS is a graph exploration algorithm that visits all the vertices of an input graph. 
Here we summarise the DFS algorithm for a graph $G=(V,E)$.
Let $S$ be a stack (initially the empty set), consisting of vertices in $V$.
Let $U$ be a set (initially $V$) of unexplored vertices in $V$, and let $X$ be a set (initially the empty set) of explored vertices in $V$.
In every step, if $S$ is empty and $|U|>0$, then we take an arbitrary element from $U$ and put it into $S$.
If the top vertex of $S$ has a neighbour in $U$, move such a neighbour from $U$ to $S$.
If the top vertex of $S$ has no neighbour in $U$, move the top vertex of $S$ to $X$. Stop when $X=V$.

The following properties hold throughout the process.
\begin{enumerate}
    \item The stack $S$ forms a path in $G$.
    \item There is no edge of $G$ between $U$ and $X$.
\end{enumerate}

\section{Finding many Hamiltonian subsets with a common vertex}\label{sec:manydiscycles}
In this section, we will prove Lemma~\ref{thr:expander-cycle} in three steps and derive Theorem~\ref{cor:cycles with v} from it. Throughout the proof, we let 
$$p= \frac{20}{\eps_1} \log n,\quad t=\frac{n}{p^{10}},\quad r = \frac{2}{\eps_1}\log^3 n.$$
We also assume that $\eps_1<1$ and $n$ is sufficiently large, such that the inequalities used in the proof are true.

First, we prove that expanders contain many disjoint cycles of appropriate length.

\subsection*{Step 1: finding many disjoint cycles}

\begin{prop}\label{prop:manydisjcycles}
    Let $H$ be an $(\eps_1,15)$-expander of order $n$.
    Then $H$ contains at least $t$ disjoint cycles, all of whose lengths are between $p^5$ and $p^6.$
\end{prop}

\begin{proof}
Let $\mathcal{C}$ be a maximal collection of disjoint cycles of length between $p^5$ and $p^6$. Suppose to the contrary that $|\mathcal{C}|<t$. 
Let $W$ be the set of vertices contained in these cycles. Note that $\abs{W} < tp^6=\frac{n}{p^4}.$
By applying Lemma~\ref{lem:Balln/10} and Lemma~\ref{lem:subgraphBall} twice, we can find two disjoint sets $D$ and $D'$ with diameter at most $p^3$ and size $\frac{n}{p^2}$ which avoid $W$.
For this, it is sufficient to note that $\frac{n}{p^2}+\frac{n}{p^4}< \eps_1 \frac{n}{20 \log^2 n}$ and thus once $D$ is constructed, one can find a large ball avoiding $W \cup D$ (by Lemma~\ref{lem:Balln/10}) and take a set $D'$ of the right size (by Lemma~\ref{lem:subgraphBall}).

Let $x=\frac n{p^2}.$ Note that $\eps(x)/4=\frac{\eps_1}{4\log^2 x} \ge \frac{\eps_1}{4\log^2 n}>\frac{1}{p^2}$ and thus $\abs{W}< \frac{n}{p^4}=\frac{x}{p^2}< x \eps(x)/4.$ Hence by Lemma~\ref{lem:avoidWshortpath} we can find a path of length at most $p^3$ connecting $D$ and $D'$ while avoiding $W$.

Now iteratively, we can build longer paths between two sets of size $\frac{n}{p^2}=tp^8$ and diameter at most $p^3$ until the length is between $p^5$ and $p^5+2p^3$, in such a way that the length increases in each step by at least one and at most $2p^3$.
Let $D$ and $D'$ be the two current sets, with a path from $v\in D$ to $v'\in D'$, say $P$.
Let $X$ be a set of size $tp^8$ and diameter at most $p^3$ which avoids $D,D'$, $W$ and $V(P)$. The latter is possible by Lemma~\ref{lem:Balln/10} and Lemma~\ref{lem:subgraphBall} as $2\frac{n}{p^2}+\frac{n}{p^4} + p^5<\eps_1 \frac{n}{20 \log^2 n}.$
Take a path of length at most $p^3$ between $X$ and $D \cup D'$ avoiding $W$ and $P$, which is again possible by Lemma~\ref{lem:avoidWshortpath}.
Without loss of generality, this path is between $x \in X$ and $ d \in D'$.
As such, we can consider $X$ and $D$ as the new sets and the union of the paths between $x$ and $d$, $d$ and $v'$ and $v'$ and $v$ as the new path $P'$. Then $|P|< |P'| \le |P| +p^3 +p^3$.

By iterating this, we reach a path $P_1$ of length between $p^5$ and $p^5+2p^3$ between the two sets $D$ and $D'$.
Applying Lemma~\ref{lem:avoidWshortpath} a final time, gives a path $P_2$, avoiding $W$ and $P_1$, between $D$ and $D'$ of length at most $p^3.$
The union of $P_1, P_2$ and $2$ connecting paths in $D$ and $D'$ gives a cycle of length between $p^5$ and $p^5+6p^3<p^6$. This cycle is disjoint from those in $\mathcal{C}$, contradicting the maximality of $\mathcal{C}$.
\end{proof}

  \definecolor{col1}{rgb}{0.5, 0.5, 0.5}
 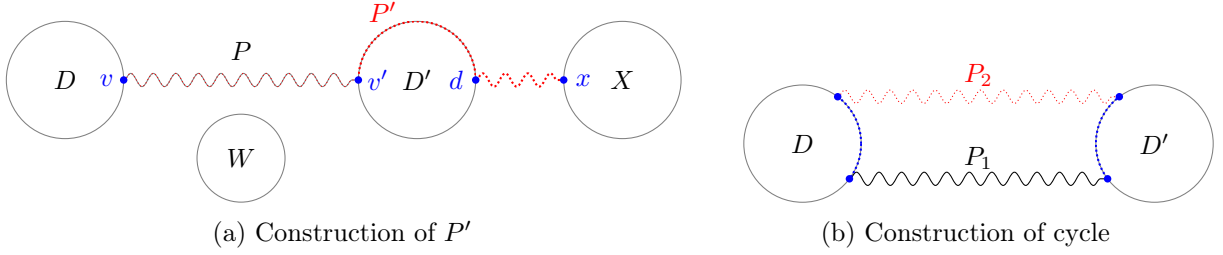
\begin{figure}[h]
\begin{minipage}[b]{0.61\textwidth}
\centering
\begin{tikzpicture}[scale=1.3]

\draw[col1] (-1,0) circle (0.6cm);
	\draw[col1] (2.6,0) circle (0.6cm);
	\draw[col1] (0.8,-0.8) circle (0.45cm);
	\draw[col1] (4.7,0) circle (0.6cm);

	\draw[decorate, decoration=snake, segment length=3mm,col1] (-0.4,0)--(2,0);
	
	\draw[red, densely dotted, decorate, decoration=snake, segment length=3mm] (-0.4,0)--(2,0);

	\draw[red,densely dotted,thick, decorate, decoration=snake, segment length=3mm] (3.2,0)--(4.1,0);

	\begin{scope}
    \clip (2,0) rectangle (3.2,0.6);
    \draw[red,densely dotted,thick] (2.6,0) circle(0.6cm);
    \end{scope}
    
  \node[inner sep= 1pt]() at (0.8,-0.8){\small $W$};
  \node[inner sep= 1pt,red]() at (2.25,0.7){\small $P'$};
  \node[inner sep= 1pt]() at (-1,0){\small $D$};
  \node[inner sep= 1pt]() at (0.8,0.3){\small $P$};
  \node[inner sep= 1pt]() at (2.6,0){\small $D'$};
  \node[inner sep= 1pt]() at (4.7,0){\small $X$};
  
  \foreach \x in{-0.4,2,3.2,4.1}
  {
  \node[inner sep= 1pt,blue]() at (\x,0)[circle,fill]{};
  }
     
	\node[inner sep= 1pt,blue](v) at (-0.575,0)[]{\small $v$};
	\node[inner sep= 1pt,blue](v') at (2.2,0)[]{\small $v'$};
	\node[inner sep= 1pt,blue]() at (3,0)[]{\small $d$};
	\node[inner sep= 1pt,blue]() at (4.3,0)[]{\small $x$};
\end{tikzpicture}
\subcaption{Construction of $P'$}
\end{minipage}
\begin{minipage}[b]{0.35\textwidth}
\centering
\begin{tikzpicture}[scale=1.3]
    \draw[col1] (-1,0) circle (0.6cm);
	\draw[col1] (2.6,0) circle (0.6cm);

	\draw[decorate, decoration=snake, segment length=3mm] (-0.52,-0.36)--(2.12,-0.36);

	\begin{scope}
    \clip (-1,-0.36) rectangle (-0.4,0.48);
    \draw[blue,densely dotted,thick] (-1,0) circle(0.6cm);
    \end{scope}

	\begin{scope}
    \clip (2,-0.36) rectangle (2.6,0.48);
    \draw[blue,densely dotted,thick] (2.6,0) circle(0.6cm);
    \end{scope}
    
    \draw[red, densely dotted, decorate, decoration=snake, segment length=3.01mm] (-0.64,0.48)--(2.24,0.48);

  \node[inner sep= 1pt,red]() at (0.8,0.7){\small $P_2$};
  \node[inner sep= 1pt]() at (-1,0){\small $D$};
  \node[inner sep= 1pt]() at (0.8,-0.15){\small $P_1$};
  \node[inner sep= 1pt]() at (2.6,0){\small $D'$};

\foreach \x in{-0.64,2.24}
  {
  \node[inner sep= 1pt,blue]() at (\x,0.48)[circle,fill]{};
  }
  
  \foreach \x in{-0.52,2.12}
  {
  \node[inner sep= 1pt,blue]() at (\x,-0.36)[circle,fill]{};
  }

\end{tikzpicture}
\subcaption{Construction of cycle}
\end{minipage}
 \caption{The sets $D,D',X,W$ and path $P$ and extension in the final step in Proposition~\ref{prop:manydisjcycles}}
    \label{fig:constr_prop3.1}
\end{figure}

\subsection*{Step 2: finding a long chain}

Having found many disjoint cycles in Step 1, we now prove that we can connect some of them in a chain. Here we use an algorithm, similar to DFS, where we explore the set of cycles (instead of set of vertices).

\begin{prop}\label{find chain}
    Let $0<\eps_1 <1$ and $H$ be an $(\eps_1,15)$-expander of order $n$.
    Let $p= \frac{20}{\eps_1} \log n, t=\frac{n}{p^{10}}$ and $r = \frac{2}{\eps_1}\log^3 n$.
    If $H$ contains at least $t$ disjoint cycles, all of whose lengths are between $p^5$ and $p^6$, then there exists an $\ell$-chain with $ \ell \ge t/ p^3$ and each path between two consecutive cycles on the chain has length at most $r=\frac{2}{\eps_1} \log^3 n$. 
\end{prop}

\begin{proof}
    Let $\mathcal{C}$ be a set consisting of $t$ disjoint cycles, all of whose lengths are between $p^5$ and $p^6$. To find the desired chain, we perform a process similar to DFS on $\C$.
    
    During the process, We keep track of the following four sets:
    \begin{itemize}
        \item a stack $S$ (initially the empty set), consisting of cycles (which are ordered) in $\mathcal{C}$,
        \item a set $U$ (initially $\mathcal{C}$) of unexplored cycles in $\mathcal{C}$,
        \item a set $X$ (initially the empty set) of explored cycles in $\mathcal{C}$,
        \item a set $\mathcal{P}$ (initially the empty set) of pairwise (vertex) disjoint paths. 
    \end{itemize}
    
    In every step, we do one of the following replacements.
    \begin{itemize}
        \item If $S$ is empty and $|U|> 0$, then take an arbitrary element from $U$ and put it into $S$.
        \item If $S$ is not empty and $|U|> 0$ and
        \begin{itemize}
            \item if there exists a path $P$ of length at most $r$, all of whose internal vertices do not belong to (a cycle of) $\mathcal{C}$ nor (a path in) $\mathcal{P}$, which connects the top element in $S$ and an arbitrary cycle $C$ in $U$, then we remove the cycle $C$ from $U$ and push it into $S$, and push the path $P$ into $\mathcal{P}$,
            \item if no such path $P$ exists, then take the top element from $S$ and put it into $X$.
        \end{itemize}
        \item If $|U|=0$, then stop.
    \end{itemize}
    Throughout the process, observe that
    \begin{enumerate}
        \item\label{DFS1} at any step of the process, there exists an $|S|$-chain which connects all cycles in $S$ and each path between two consecutive cycles has length at most $r$;
        \item there does not exist a path $P$, whose internal vertices are not in $\mathcal{C}$ and $\mathcal{P}$, with length at most $r$, which connects (a cycle in) $X$ and (a cycle in) $U$; and
        \item $|\mathcal{P}| \le |\mathcal{C}| - |U|$.
    \end{enumerate}

    We run this process until the point that $|U| = \frac t3$.
    Suppose to the contrary that $|S|< t/p^3$.
    Then $|X|=t-|U|-|S| \ge t/3$ and $|\mathcal{P}|\le 2t/3$. 
    Observe also that
    \begin{align*}
        &\sum_{C\in X} |C| ,   \sum_{C\in U} |C| \ge t/3 \cdot p^5 ,\\
        &\sum_{P\in\mathcal{P}} |P| \le 2t/3 \cdot r = \frac{4t}{3\eps_1}
    \log^3 n \le  \frac{\eps_1tp^5}{24 \log ^2 n}  \le \eps(tp^5/3)tp^5/24 \quad \mbox{and}\\
    &\sum_{ C \in S } |C| \le  tp^3 \le \eps(tp^5/3)tp^5/24.
    \end{align*}
    Let $W =\bigcup_{P\in \mathcal{P} } V(P) \cup \bigcup_{C\in S} V(C)$, then $ |W| \le \eps(tp^5/3)tp^5/12$.
    By Lemma~\ref{lem:avoidWshortpath}, there exists a path avoiding vertices in $W$ with length at most $r$ between $X$ and $U$, a contradiction. Hence $|S|\ge t/p^3$ and Observation~\ref{DFS1} above implies the result.
\end{proof}

\subsection*{Step 3: finding a long wheel}

Finally, we prove that one can add a path between two cycles (near the ends) of the chain, to find a wheel.

\begin{prop}\label{prop:wheel}
    Let $0<\eps_1 <1$ and $H$ be an $(\eps_1,15)$-expander of order $n$.
    Let $p= \frac{20}{\eps_1} \log n, t=\frac{n}{p^{10}}$, and $r = \frac{2}{\eps_1}\log^3 n$.
    If $H$ contains an $m$-chain with $m= t/p^3$ such that each cycle has length between $p^5$ and $p^6$ and each path between two consecutive cycles has length at most $r$.
    Then $H$ contains an $\ell$-wheel with $\ell \ge 2t/p^6$.
\end{prop}
\begin{proof}
    Let $S= C_1P_1C_2\ldots P_{m-1}C_m$ be the $m$-chain with $|P_i|\le r$.
    Let 
    \begin{align*}
        &X_1= \bigcup_{i=1}^{m/2-m/p^3} V(C_i), 
        \ W = \left( \bigcup_{i=m/2-m/p^3+1}^{m/2+m/p^3} V(C_i) \right) \cup \bigcup_{i=1}^{m} V(P_i),
        \  X_2= \bigcup_{i=m/2+m/p^3+1}^{m} V(C_i).
    \end{align*}
    It is easy to check $|X_i| \ge m/3\cdot p^5 $ for $i=1,2$, and $|W| \le  2m/p^3 \cdot p^6 + m \cdot p^3  \le \eps( mp^5/3)mp^5/12$.
    By Lemma~\ref{lem:avoidWshortpath}, there exists a path $P$ avoiding $W$ between $X_1$ and $X_2$, say between $C_i$ and $C_j$ for some $1\le i\le m/2-m/p^3$ and $m/2+m/p^3+1\le j\le m$.
    Then the union of $P$ and $C_iP_iC_{i+1}\ldots P_{j-1}C_j$
    forms an $\ell$-wheel with $ \ell \ge 2m/p^3 = 2t/p^6$, where possibly $C_i$ and/or $C_j$ are deleted when $P$ has an end-vertex equal to $C_i \cap P_i$ and/or $C_{j} \cap P_{j-1}$.
\end{proof}
So we can find an $\ell$-wheel with $ \ell \ge 2n/p^{16} =(\frac{\eps_1}{20})^{16} (2n/\log^{16} n) $ in $H$. Lemma~\ref{thr:expander-cycle} now follows immediately from Propositions~\ref{prop:manydisjcycles},~\ref{find chain} and~\ref{prop:wheel}.

\begin{proof}[Proof of Theorem~\ref{cor:cycles with v}]
    Take $C=30\log3$ and $\eps_1 = 1/(10C)$.
    By Lemma~\ref{lem:existence_expander} there exists an $(\eps_1, 15)$-expander $H \subset G$ of order $n$ for which $\abs{H} \ge t=c_{\alpha}(G) \ge \a d$, as $d(H) \ge \frac{1}{2} d(G)$.
    Then by Lemma~\ref{thr:expander-cycle} there exists an $\ell$-wheel in $G$ with $\ell \ge \left(\frac{\eps_1}{20} \right)^{16} (2n/\log^{16} n)  $. Since an $\ell$-wheel contains $2^\ell$ different cycles and any vertex in an $\ell$-wheel is contained in at least $2^{\ell-1}$ different cycles, we conclude.
\end{proof}

\section{ Proof of the main theorem}\label{sec:mainthr}
In this section, we prove Theorem~\ref{main:1}. We shall perform three counting strategies and show that at least one results in the desired lower bound for the number of Hamiltonian subsets. We start with a smallest counterexample $G$ and in the first two strategies, we shall find many vertices belonging to many (different) Hamiltonian subsets. If those strategies fail to produce enough Hamiltonian subsets, then $G$ must contain a dense subgraph with sufficiently Hamiltonian subsets.

\begin{proof}[Proof of Theorem~\ref{main:1}]
    Recall that $t=c_\alpha(G)$, $\alpha=1/5$, $\eps_1=\frac{1}{300\log3}$ and $\b=\left( \frac{\eps_1}{20} \right)^{16}$.
    Choose $d_0$ such that $\b x> \log^{16} x$ whenever $x \ge \a d_0$ and Corollary~\ref{cor:cycles with v} is true whenever $d \ge d_0$. Then choose the constant $B$ such that 
    $\b bt/\log^{16}(bt)> \b t/\log^{16}t + \log_2(5b)$ whenever $b \ge \frac{B}{5}$ and $t \ge \a d_0.$
    
    Let $G$ be a graph with average degree $d\ge d_0$ with minimum order $n$ among all graphs for which the theorem is not true. That is, $h(G)\ge \frac{n}{B} 2^{\b t/\log^{16}t}$, and for any proper subgraph of $G$, say $G'$, if the average degree of $G'$ is at least $d_0$, then $h(G')\ge \frac{|V(G')|}{B} 2^{\b t'/\log^{16}t'}$, where $t'=c_\alpha(G')$.
    
    We now consider three strategies.

    \medskip
    
    \noindent\textbf{Strategy 1:}
    We will choose a set of vertices $S$ in which every vertex belongs to at least $2^{\b t/\log^{16} t}$ different Hamiltonian subsets of $G$.
    Let $S$ be a set of vertices, which is initially the empty set.
    As long as $G - S$ has average degree at least $\frac d4,$ by Lemma~\ref{lem:existence_expander} there exists an $(\eps_1,15)$-expander $H \subset G- S$ with minimum degree at least $\frac d5$. So $H$ has order at least $c_{\a}(G).$
    By Corollary~\ref{cor:cycles with v}, there is a vertex $s$ in $H$ belonging to at least $2^{\b t/\log^{16} t}$ distinct Hamiltonian subsets.
    Now add $s$ to $S$.
    If at the end, $S$ contains at least $\frac{n}{B}$ vertices, we would reach a contradiction with the choice of $G$, being a graph with less than $\frac{n}{B} 2^{\b t/\log^{16} t}$ Hamiltonian subsets.
    
    We may then assume that $|S|<\frac{n}{B}$. Note also that $G-S$, having average degree less than $\frac d4$, contains less than $\frac{d(n-\abs{S})}{8}< \frac{m}{4}$ edges, where $m=e(G)$.
    
    \medskip
    
    \noindent\textbf{Strategy 2:}
    We restart the search for Hamiltonian subsets in the bipartite graph $G[S, V \backslash S]$.
    As long as the bipartite graph $G[S, V \backslash S]$ contains at least $\frac{m}{4}$ edges, and hence has average degree at least $\frac d4,$ by Lemma~\ref{lem:existence_expander} there exists an $(\eps_1,15)$-expander $H \subset G[S, V \backslash S]$ with minimum degree at least $\frac d5$ and hence $H$ has order at least $c_{\a}(G).$
    Take a vertex $s \in V \backslash S$,
    then by Corollary~\ref{cor:cycles with v} there are at least 
    $2^{\b t/\log^{16}t}$ Hamiltonian subsets containing $s.$
    Now add $s$ to $S$. 
    If $G[S, V \backslash S]$ contains at least $\frac{m}{4}$ edges, repeat this process.

    We claim that we do not count the same Hamiltonian subset twice. 
    For two different vertices $s_1$ and $s_2$ which we add to $S$ in strategy 2, let $S_i$ be the set $S$ before we move $s_i$, for $i=1,2$, then $s_i \notin S_i$. 
    We may assume $S_1\subseteq S_2$, then $s_1\in S_2$ and $s_2 \notin S_1$.
    Let $H$ be an Hamiltonian subset containing both $s_1$ and $s_2$.
    Since $s_1 \notin S_1, s_1 \in S_2$ and $s_1 \in H$, we have $|H \cap S_1|<|H \cap S_2|$, a contradiction to $|H \cap S_1| = |H \cap S_2| = |H|/2$.    
    If we can repeat this at least $\frac{n}{B}$ times, we have found the desired number of Hamiltonian subsets again, which would be the desired contradiction.

    \medskip
    
    \noindent\textbf{Strategy 3:}
    After performing the two previous strategies, we ended up with a set $S$ such that $G- S$ and $G[S, V \backslash S]$ both contain less than $\frac{m}{4}$ edges, so $G[S]$ contains at least $\frac{m}{2}$ edges. Also we know that $\abs{S}=\frac{1}{b} n \le 2\frac{n}{B}, $ for some $b \ge \frac B2.$
    Hence the average degree of $G[S]$ is $\gamma d$
    for some $\gamma \ge \frac b2.$
    By Proposition~\ref{prop:ca(G)~an}, this implies that 
    $ \bceil{\frac{1}{\gamma}( c_\a(G[S])-1) +1} \ge c_{\a/\gamma}(G[S]) \ge c_{\a}(G)$
    and thus $\gamma (c_\a(G)-2)<c_\a(G[S]) = b' c_{\a}(G)=b't $ for some $b' \ge \frac{4\gamma}{5} \ge \frac{2b}{5} \ge \frac B5$.
    Since $G$ is a minimal counterexample, $G[S]$ satisfies
    \begin{align*}
        h(G[S]) &\ge \frac{n/b}{B} 2^{\b b't/\log^{16}(b't)}
        \ge 5b' \frac{n/b}{B} 2^{\b t/\log^{16}t}
        \ge \frac{n}{B} 2^{\b t/\log^{16}t}.
    \end{align*}
    Since $h(G) \ge h(G[S])$, we derive the final contradiction.
\end{proof}

\section{Proof of Proposition~\ref{coro:3}}\label{sec:proofbgraph}
We shall prove a similar bound for a larger class of $\b$-graphs (see~\cite{FK21}). For $0<\b<1$, an $n$-vertex graph $G$ is an \emph{$\b$-graph} if every pair of disjoint vertex sets $A,B\subseteq V(G)$ of size $|A|,|B|> \b n$ are connected by an edge. Note that by the expander mixing lemma, Lemma~\ref{EML}, $(n,d,\lambda)$-graph with $\frac{d}{|\lambda|} \ge \frac{1}{\b ^2}$ is an $\b$-graph. Hence, Proposition~\ref{coro:3} follows from the following result.

\begin{prop}\label{main:3}
Let $G$ be an $n$-vertex $\b$-graph, then $ h(G) \ge  \binom{n}{\frac{1}{2}n} / \binom{(\frac{1}{2}+ 3 \b)n}{3\b n}$.  Specially, if $\b < 1/6$, then $h(G) \ge  2^{(\frac{1}{2}-3\b)n}$.
\end{prop}

We first show that large subgraphs of an $\b$-graph contain almost spanning cycles.

\begin{lem}\label{Cycle in beta graph}
    Let $G$ be an $\b$-graph and $0<c<1$. Then for every subset $S$ with $cn$ vertices, $G[S]$ contains a cycle of length at least $(c-3\b)n.$
\end{lem}

To prove this lemma, we need the following result.
\begin{lem}[\cite{Krivelevich2019}, Theorem 1]\label{lem:find long cycle}
    Let $k>0,t\ge 2$ be integers. Let $G$ be a graph on more than $k$ vertices, satisfying that $|N_G(W)| \ge t$, for every $W \subseteq V(G)$ with $k/2 \le |W| \le k$. Then $G$ contains a cycle of length at least $t+1$.
\end{lem}

\begin{proof}[Proof of Lemma~\ref{Cycle in beta graph}]
By definition of an $\b$-graph, for any vertex set $U$, if $\abs{U}\ge \b n$, then we have $\abs{N_G(U)}> n-|U| -\b n$, for otherwise $U$ and $V \backslash(U\cup N(U))$ would be two sets of size at least $\b n$ without an edge in between, a contradiction.
Let $W$ be any subset of $S$ of size $\b n  \le \abs{W} \le 2\b n$.
Then $\abs{N_{G[S]}(W)} \ge \abs{N_G(W)} - (n-\abs{S}) \ge n- 3 \b n-(1-c)n=(c-3\b)n$.
Now the result follows from Lemma~\ref{lem:find long cycle}.
\end{proof}


\begin{proof}[Proof of Proposition~\ref{main:3}]
    
By Lemma~\ref{Cycle in beta graph}, for any vertex set $S$ of size $\frac{1}{2}n$, we can find a cycle of length at least $(\frac{1}{2} -3\b )n$. 
For any such cycle $C_\ell$, with $(\frac{1}{2} -3\b )n \le \ell \le n/2$, it is contained in at most $\binom{n-\ell}{n/2-\ell} \le \binom{\left(\frac{1}{2}+ 3 \b\right)n}{3\b n} $ different subsets of size $n/2$. 
    So we can find at least $\binom{n}{\frac{1}{2}n} / \binom{\left(\frac{1}{2}+ 3 \b\right)n}{3\b n}$ Hamiltonian sets.

    If $\b < 1/6$, we have
    $$ \binom{n}{\frac{1}{2}n} / \binom{\left(\frac{1}{2}+ 3 \b\right)n}{3\b n} =\frac{n\cdot \left(n-1\right) \cdot \ldots  \cdot \left(n-\left(1/2 -3\b\right)n +1\right) }{\frac{1}{2} n \cdot \left(\frac{1}{2} n -1\right) \cdot \ldots \cdot \left(\frac{1}{2}n-\left(1/2 -3\b\right)n +1\right)}\ge 2^{\left(\frac{1}{2}-3\b\right)n}.$$
\end{proof}

\section{Concluding remarks}
\label{sec:conc}
In this paper, we proved a near optimal lower bound on the number of Hamiltonian subsets in a graph with given minimum degree, which asymptotically gives much better bounds for large graphs. Kim, Liu, Sharifzadeh and Staden~\cite[Thr.~1.3]{KLSS17} proved that for $d$ sufficiently large, any graph $G$ different from $K_{d+1}$ with minimum degree $\delta(G) \ge d$ has at least roughly twice as many Hamiltonian subsets as $K_{d+1}$. The following extension of Koml\'{o}s conjecture seems plausible.

\begin{prob}
    Let $d\ge 3$ be an odd integer.
    Let $G$ be a graph different from $K_{d+1}$ with minimum degree $\delta(G) \ge d$.
    Is $h(G) \ge 2 h(K_{d+1})$?
\end{prob}

Equality occurs if $G\in \{ 2K_{d+1}, K_{d+1} \star K_{d+1}, K_{d+2} \backslash M\},$ where $M$ is a perfect matching of $K_{d+2},$ or when $G=K_{3,3}$ and $d=3$.
Here $K_{d+1} \star K_{d+1}$ is the union of two $K_{d+1}$s which are vertex-disjoint except from one common vertex. Notice that the same is not true for even $d$, as then 
$$h(K_{d+2} \backslash M) = 2^{d+2}-d^2-\frac72 d -4 <2 h(K_{d+1})=2^{d+2}-d^2-3d -4.$$


\bibliographystyle{abbrv}
\bibliography{crux}

\begin{thebibliography}{10}

\bibitem{AC88}
N.~Alon and F.~R.~K. Chung.
\newblock Explicit construction of linear sized tolerant networks.
\newblock In {\em Proceedings of the {F}irst {J}apan {C}onference on {G}raph
  {T}heory and {A}pplications ({H}akone, 1986)}, volume~72, pages 15--19, 1988.

\bibitem{csaba2016proof}
B.~Csaba, D.~K{\"u}hn, A.~Lo, D.~Osthus, and A.~Treglown.
\newblock {\em Proof of the 1-factorization and Hamilton decomposition
  conjectures}, volume 244.
\newblock American Mathematical Society, 2016.

\bibitem{dirac1952some}
G.~A. Dirac.
\newblock Some theorems on abstract graphs.
\newblock {\em Proceedings of the London Mathematical Society}, 3(1):69--81,
  1952.

\bibitem{FKKl22}
I.~G. Fern\'{a}ndez, J.~Kim, Y.~Kim, and H.~Liu.
\newblock Nested cycles with no geometric crossings.
\newblock {\em Proc. Amer. Math. Soc. Ser. B}, 9:22--32, 2022.

\bibitem{FL22}
I.~G. Fern{\'a}ndez and H.~Liu.
\newblock How to build a pillar: a proof of {T}homassen's conjecture.
\newblock {\em arXiv preprint}, arXiv: 2201.07777.

\bibitem{FK21}
L.~Friedman and M.~Krivelevich.
\newblock Cycle lengths in expanding graphs.
\newblock {\em Combinatorica}, 41(1):53--74, 2021.

\bibitem{Fer22}
I.~{Gil Fern{\'a}ndez}, J.~{Hyde}, H.~{Liu}, O.~{Pikhurko}, and Z.~{Wu}.
\newblock {Disjoint isomorphic balanced clique subdivisions}.
\newblock {\em arXiv preprint}, arXiv: 2204.12465.

\bibitem{HHKLLW22}
J.~Haslegrave, J.~Hu, J.~Kim, H.~Liu, B.~Luan, and G.~Wang.
\newblock Crux and long cycles in graphs.
\newblock {\em SIAM J. Discrete Math.}, 36(4):2942--2958, 2022.

\bibitem{HHKL22}
J.~{Haslegrave}, J.~{Hyde}, J.~{Kim}, and H.~{Liu}.
\newblock {Ramsey numbers of cycles versus general graphs.}
\newblock {\em Forum of Mathematics, Sigma}, to appear. arXiv: 2112.03893.

\bibitem{HKL22}
J.~Haslegrave, J.~Kim, and H.~Liu.
\newblock Extremal density for sparse minors and subdivisions.
\newblock {\em Int. Math. Res. Not. IMRN}, (20):15505--15548, 2022.

\bibitem{Im22}
S.~{Im}, J.~{Kim}, Y.~{Kim}, and H.~{Liu}.
\newblock {Crux, space constraints and subdivisions}.
\newblock {\em arXiv preprint}, arXiv: 2207.06653.

\bibitem{karp1972reducibility}
R.~M. Karp.
\newblock Reducibility among combinatorial problems.
\newblock In {\em Complexity of computer computations}, pages 85--103.
  Springer, 1972.

\bibitem{KLSS17}
J.~Kim, H.~Liu, M.~Sharifzadeh, and K.~Staden.
\newblock Proof of {K}oml\'{o}s's conjecture on {H}amiltonian subsets.
\newblock {\em Proc. Lond. Math. Soc. (3)}, 115(5):974--1013, 2017.

\bibitem{knox2015edge}
F.~Knox, D.~K{\"u}hn, and D.~Osthus.
\newblock Edge-disjoint hamilton cycles in random graphs.
\newblock {\em Random Structures \& Algorithms}, 46(3):397--445, 2015.

\bibitem{KS94}
J.~Koml\'{o}s and E.~Szemer\'{e}di.
\newblock Topological cliques in graphs.
\newblock {\em Combin. Probab. Comput.}, 3(2):247--256, 1994.

\bibitem{KS96}
J.~Koml\'{o}s and E.~Szemer\'{e}di.
\newblock Topological cliques in graphs. {II}.
\newblock {\em Combin. Probab. Comput.}, 5(1):79--90, 1996.

\bibitem{Krivelevich2019}
M.~Krivelevich.
\newblock Long cycles in locally expanding graphs, with applications.
\newblock {\em Combinatorica}, 39(1):135--151, 2019.

\bibitem{kuhn2013hamilton}
D.~K\"{u}hn and D.~Osthus.
\newblock Hamilton decompositions of regular expanders: a proof of {K}elly's
  conjecture for large tournaments.
\newblock {\em Adv. Math.}, 237:62--146, 2013.

\bibitem{kuhn2014hamiltonsurvey}
D.~K\"{u}hn and D.~Osthus.
\newblock Hamilton cycles in graphs and hypergraphs: an extremal perspective.
\newblock In {\em Proceedings of the {I}nternational {C}ongress of
  {M}athematicians---{S}eoul 2014. {V}ol. {IV}}, pages 381--406. Kyung Moon Sa,
  Seoul, 2014.

\bibitem{kuhn2014hamilton}
D.~K{\"u}hn and D.~Osthus.
\newblock Hamilton decompositions of regular expanders: applications.
\newblock {\em Journal of Combinatorial Theory, Series B}, 104:1--27, 2014.

\bibitem{Kuhn2005}
D.~K\"{u}hn, D.~Osthus, and A.~Taraz.
\newblock Large planar subgraphs in dense graphs.
\newblock {\em J. Combin. Theory Ser. B}, 95(2):263--282, 2005.

\bibitem{LM17}
H.~Liu and R.~Montgomery.
\newblock A proof of {M}ader's conjecture on large clique subdivisions in
  {$C_4$}-free graphs.
\newblock {\em J. Lond. Math. Soc. (2)}, 95(1):203--222, 2017.

\bibitem{LM22}
H.~Liu and R.~Montgomery.
\newblock A solution to {E}rd{\H{o}}s and {H}ajnal's odd cycle problem.
\newblock {\em Journal of the American Mathematical Society}, to appear. DOI:
  https://doi.org/10.1090/jams/1018.

\bibitem{LWY22}
H.~Liu, G.~Wang, and D.~Yang.
\newblock Clique immersion in graphs without a fixed bipartite graph.
\newblock {\em J. Combin. Theory Ser. B}, 157:346--365, 2022.

\end{thebibliography}

\end{document}